\documentclass[reqno]{amsart}
\usepackage[dvips]{graphicx}
\newtheorem{thm}{Theorem}
\newtheorem{lem}{Lemma}          
\newtheorem{defn}{Definition}          
\begin{document}
\AmS -\TeX
\title[Constructive proof of Brouwer's fixed point theorem]{Constructive proof of Brouwer's fixed point theorem for sequentially locally non-constant and uniformly sequentially continuous functions by Sperner's lemma}

\author{Yasuhito Tanaka}
\thanks{This work was supported in part by the Ministry of Education, Science, Sports and Culture of Japan, Grant-in-Aid for Scientific Research (C), 20530165, and the Special Costs for Graduate Schools of the Special Expenses for Hitech Promotion by the Ministry of Education, Science, Sports and Culture of Japan in 2010.}\address{Faculty of Economics, Doshisha University, Kamigyo-ku, Kyoto, 602-8580, Japan}
\email{yasuhito@mail.doshisha.ac.jp}


\begin{abstract}
In this paper using Sperner's lemma for modified partition of a simplex we will constructively prove Brouwer's fixed point theorem for sequentially locally non-constant and uniformly sequentially continuous functions. We follow the Bishop style constructive mathematics according to \cite{bb}, \cite{br} and \cite{bv}.
\end{abstract}

\maketitle

\section{Introduction}
It is well known that Brouwer's fixed point theorem can not be constructively proved in general case. Sperner's lemma which is used to prove Brouwer's theorem, however, can be constructively proved. Some authors, for example \cite{da} and \cite{veld}, have presented a constructive (or an approximate) version of Brouwer's theorem using Sperner's lemma. Thus, Brouwer's fixed point theorem can be constructively proved in its constructive version. Also Dalen in \cite{da} states a conjecture that a uniformly continuous function $f$ from a simplex to itself, with property that each open set contains a point $x$ such that $x\neq f(x)$ and also for every point $x$ on the faces of the simplex $x\neq f(x)$, has an exact fixed point. We call such a property \emph{local non-constancy}. Further we define a stronger property \emph{sequential local non-constancy}. In another paper \cite{ta1} we have constructively proved Dalen's conjecture with sequential local non-constancy. But in that paper as with \cite{da} and \cite{veld} we assume uniform continuity of functions. We consider a weaker uniform sequential continuity of functions according to \cite{br2}. In classical mathematics uniform continuity and uniform sequential continuity are equivalent. In constructive mathematics a la Bishop, however, uniform sequential continuity is weaker than uniform continuity\footnote{Also in constructive mathematics sequential continuity is weaker than continuity, and uniform continuity (respectively, uniform sequential continuity) is stronger than continuity (respectively, sequential continuity) even in a compact space. See, for example, \cite{ishi}. As stated in \cite{bm} all proofs of the equivalence between continuity and sequential continuity involve the law of excluded middle, and so the equivalence of them is non-constructive.} In this paper using Sperner's lemma for a modified partition of a simplex we will constructively prove Dalen's conjecture for sequentially locally non-constant and uniformly sequentially continuous functions.

In the next section we prove a modified version of Sperner's lemma. In Section 3 we present a proof of Brouwer's fixed point theorem for sequentially locally non-constant and uniformly sequentially continuous functions by the modified version of Sperner's lemma. We follow the Bishop style constructive mathematics according to \cite{bb}, \cite{br} and \cite{bv}.

\section{Sperner's lemma}\label{sec2}
To prove Sperner's lemma we use the following simple result of graph theory, Handshaking lemma\footnote{For another constructive proof of Sperner's lemma, see \cite{su}. }. A \emph{graph} refers to a collection of vertices and a collection of edges that connect pairs of vertices. Each graph may be undirected or directed. Figure \ref{graph} is an example of an undirected graph. Degree of a vertex of a graph is defined to be the number of edges incident to the vertex, with loops counted twice. Each vertex has odd degree or even degree. Let $v$ denote a vertex and $V$ denote the set of all vertices.
\begin{lem}[Handshaking lemma]
Every undirected graph contains an even number of vertices of odd degree. That is, the number of vertices that have an odd number of incident edges must be even.
\end{lem}
This is a simple lemma. But for completeness of arguments we provide a proof.
\begin{proof}
Prove this lemma by double counting. Let $d(v)$ be the degree of vertex $v$. The number of vertex-edge incidences in the graph may be counted in two different ways: by summing the degrees of the vertices, or by counting two incidences for every edge. Therefore,
\[\sum_{v\in V}d(v)=2e,\]
where $e$ is the number of edges in the graph. The sum of the degrees of the vertices is therefore an even number. It could happen if and only if an even number of the vertices had odd degree.

\end{proof}

\begin{figure}[ptb]
\begin{center}
\includegraphics[height=5cm]{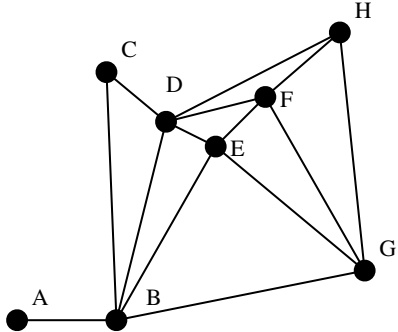}
\end{center}
	\vspace*{-.3cm}
	\caption{Example of graph}
	\label{graph}
\end{figure}

Let $\Delta$ denote an $n$-dimensional simplex. $n$ is a finite natural number. For example, a 2-dimensional simplex is a triangle.  Let partition or triangulate a simplex. Figure \ref{tria1} is an example of partition (triangulation) of a 2-dimensional simplex. In a 2-dimensional case we divide each side of $\Delta$ in $m$ equal segments, and draw the lines parallel to the sides of $\Delta$. Then, the 2-dimensional simplex is partitioned into $m^2$ triangles. We consider partition of $\Delta$ inductively for cases of higher dimension. In a 3 dimensional case each face of $\Delta$ is an 2-dimensional simplex, and so it is partitioned into $m^2$ triangles in the way above mentioned, and draw the planes parallel to the faces of $\Delta$. Then, the 3-dimensional simplex is partitioned into $m^3$ trigonal pyramids. And so on for cases of higher dimension.

\begin{figure}[ptb]
\begin{center}
\includegraphics[height=7.5cm]{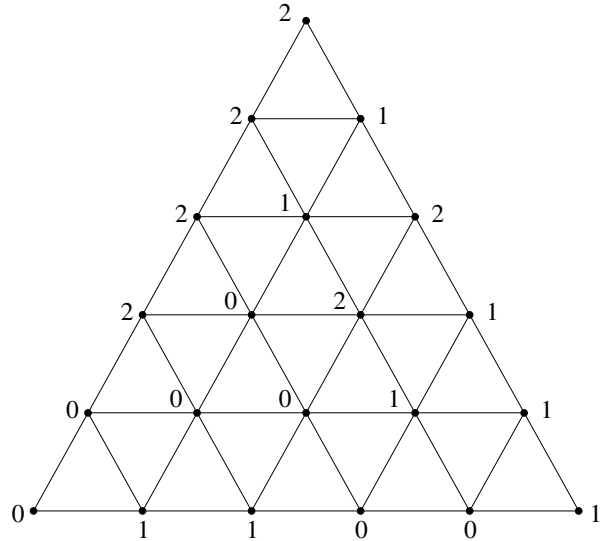}
\end{center}
	\vspace*{-.3cm}
	\caption{Partition and labeling of 2-dimensional simplex}
	\label{tria1}
\end{figure}

Let $K$ denote the set of small $n$-dimensional simplices of $\Delta$ constructed by partition. Vertices of these small simplices of $K$ are labeled with the numbers 0, 1, 2, $\dots$, $n$ subject to the following rules.
\begin{enumerate}
\item The vertices of $\Delta$ are respectively labeled with 0 to $n$. We label a point $(1,0, \dots, 0)$ with 0, a point $(0,1,0, \dots, 0)$ with 1, a point $(0,0,1 \dots, 0)$ with 2, $\dots$, a point $(0,\dots, 0,1)$ with $n$. That is, a vertex whose $k$-th coordinate ($k=0, 1, \dots, n$) is $1$ and all other coordinates are 0 is labeled with $k$. 

\item If a vertex of $K$ is contained in an $n-1$-dimensional face of $\Delta$, then this vertex is labeled with some number which is the same as the number of a vertex of that face.

\item If a vertex of $K$ is contained in an $n-2$-dimensional face of $\Delta$, then this vertex is labeled with some number which is the same as the number of a vertex of that face. And so on for cases of lower dimension.

\item A vertex contained inside of $\Delta$ is labeled with an arbitrary number among 0, 1, $\dots$, $n$.
\end{enumerate}

Now we modify this partition of a simplex as follows.
\begin{quote}
Put a point in an open neighborhood around each vertex inside $\Delta$, and make partition of $\Delta$ replacing each vertex inside $\Delta$ by that point in each neighborhood. The diameter of each neighborhood should be sufficiently small relatively to the size of each small simplex. We label the points in $\Delta$ following the rules (1) $\sim$ (4).
\end{quote}
Then, we obtain a partition of $\Delta$ illustrated in Figure \ref{tria12}.

We further modify this partition as follows;
\begin{quote}
Put a point in an open neighborhood around each vertex on a face (boundary) of $\Delta$, and make partition of $\Delta$ replacing each vertex on the face by that point in each neighborhood, and we label the points in $\Delta$ following the rules (1) $\sim$ (4). This neighborhood is open in a space with dimension lower than $n$.
\end{quote}
Then, we obtain a partition of $\Delta$ depicted in Figure \ref{tria13}.

\begin{figure}[ptb]
\begin{center}
\includegraphics[height=7.5cm]{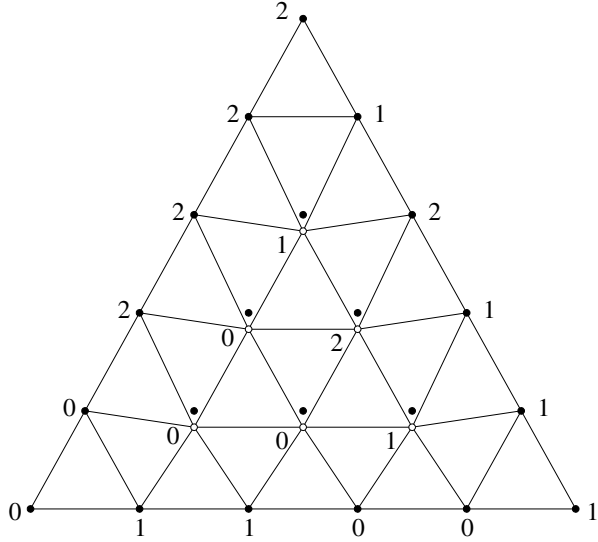}
\end{center}
	\vspace*{-.3cm}
	\caption{Modified partition of a simplex}
	\label{tria12}
\end{figure}

A small simplex of $K$ in this modified partition which is labeled with the numbers 0, 1, $\dots$, $n$ is called a \emph{fully labeled simplex}. Now let us prove Sperner's lemma about the modified partition of a simplex.
\begin{lem}[Sperner's lemma]
If we label the vertices of $K$ following above rules (1) $\sim$ (4), then there are an odd number of fully labeled simplices. Thus, there exists at least one fully labeled simplex. \label{l2}
\end{lem}
\begin{proof}
See Appendix \ref{app1}.
\end{proof}

\begin{figure}[t]
\begin{center}
\includegraphics[height=8cm]{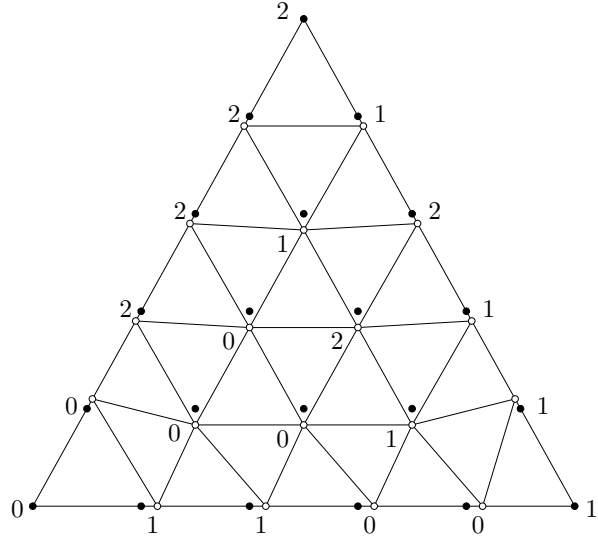}
\end{center}
	\vspace*{-.3cm}
	\caption{Modified partition of a simplex: Two}
	\label{tria13}
\end{figure}

\section{Brouwer's fixed point theorem for sequentially locally non-constant and uniformly sequentially continuous functions}

Let $x=(x_0, x_1, \dots, x_n)$ be a point in an $n$-dimensional simplex $\Delta$, and consider a function $f$ from $\Delta$ to itself.  Denote the $i$-th components of $x$ and $f(x)$ by, respectively, $x_i$ and $f_i(x)$ or $f_i$. 

Uniform continuity, sequential continuity and uniform sequential continuity of functions are defined as follows.
\begin{defn}[Uniform continuity]
A function $f$ is uniformly continuous in $\Delta$ if for any $x, x'\in \Delta$ and $\varepsilon>0$ there exists $\eta>0$ such that
\[\mathrm{If}\ |x-x'|<\eta,\mathrm{\ then}\ |f(x)-f(x')|<\varepsilon.\]
$\eta$ depends on only $\varepsilon$.
\end{defn}
\begin{defn}[Sequential continuity]
A function $f$ is sequentially continuous at $x\in \Delta$ in $\Delta$ if for sequences $(x_n)_{n\geq 1}$ and $(f(x_n))_{n\geq 1}$ in $\Delta$
\[f(x_n)\longrightarrow f(x)\ \mathrm{whenever}\ x_n\longrightarrow x.\]
\end{defn}
\begin{defn}[Uniform sequential continuity]
A function $f$ is uniformly sequentially continuous in $\Delta$ if for sequences $(x_n)_{n\geq 1}$, $(x'_n)_{n\geq 1}$, $(f(x_n))_{n\geq 1}$ and $(f(x'_n))_{n\geq 1}$ in $\Delta$
\[|f(x_n)-f(x'_n)|\longrightarrow 0\ \mathrm{whenever}\ |x_n-x'_n|\longrightarrow 0.\]
\end{defn}
$|x_n-x'_n|\longrightarrow 0$ means
\[\forall \varepsilon>0\ \exists N\ \forall n\geq N\ (|x_n-x'_n|<\varepsilon),\]
where $\varepsilon$ is a real number, and $n$ and $N$ are natural numbers. Similarly, $|f(x_n)-f(x'_n)|\longrightarrow 0$ means
\[\forall \varepsilon>0\ \exists N'\ \forall n\geq N'\ (|f(x_n)-f(x'_n)|<\varepsilon).\]
$N'$ is a natural number. In classical mathematics uniform continuity and uniform sequential continuity of functions are equivalent. But in constructive mathematics a ala Bishop uniform sequential continuity is weaker than uniform continuity, and uniform sequential continuity is stronger than sequential continuity. 

On the other hand, the definition of local non-constancy of functions is as follows:
\begin{defn}[Local non-constancy of functions]
\begin{enumerate}
	\item At a point $x$ on a boundary of a simplex $f(x)\neq x$. This means $f_i(x)>x_i$ or $f_i(x)<x_i$ for at least one $i$. 
	\item In any open set of $\Delta$ there exists a point $x$ such that $f(x)\neq x$.
\end{enumerate}
\end{defn}
We define modified local non-constancy of functions as follows;
\begin{defn}[Modified local non-constancy of functions]
\begin{enumerate}
\item At the vertices of a simplex $\Delta$ $f(x)\neq x$.
\item In any open set contained in the faces (boundaries) of $\Delta$ there exists a point $x$ such that $f(x)\neq x$. This open set is open in a space of dimension lower than $n$.
\item In any open set of $\Delta$ there exists a point $x$ such that $f(x)\neq x$.
\end{enumerate}
\end{defn}
(2) of the modified local non-constancy implies that every vertex $x$ in a partition of a simplex, for example, as illustrated by white circles in Figure \ref{tria13} in a 2-dimensional case, can be selected to satisfy $f(x)\neq x$ even when points on the faces of $\Delta$ (black circles on the edges) do not necessarily satisfy this condition. Even if a function $f$ does not strictly satisfy the local non-constancy so long as it satisfies the modified local non-constancy, we can partition $\Delta$ to satisfy the conditions for Sperner's lemma.

Next, by reference to the notion of \emph{sequentially at most one maximum} in \cite{berg}, we define the property of \emph{sequential local non-constancy}. First we recapitulate the compactness (total boundedness with completeness) of a set in constructive mathematics. $\Delta$ is compact in the sense that for each $\varepsilon>0$ there exists a finite $\varepsilon$-approximation to $\Delta$. An $\varepsilon$-approximation to $\Delta$ is a subset of $\Delta$ such that for each $x\in \Delta$ there exists $y$ in that $\varepsilon$-approximation with $|x-y|<\varepsilon$. Each face (boundary) of $\Delta$ is also a simplex, and so it is compact in a space with dimension lower than $n$. The definition of sequential local non-constancy is as follow;
\begin{defn}[Sequential local non-constancy of functions]
\begin{enumerate}
\item At the vertices of a simplex $\Delta$ $f(x)\neq x$.
\item There exists $\bar{\varepsilon}>0$ with the following property. We have a finite $\varepsilon$-approximation $L=\{x^1, x^2, \dots, x^l\}$ to each face of $\Delta$ for each $\varepsilon$ with $0<\varepsilon<\bar{\varepsilon}$ such that if for all sequences $(x_m)_{m\geq 1}$, $(y_m)_{m\geq 1}$ in each open $\varepsilon$-ball $S'$, which is a subset of the face, around each $x^i\in L$ $|f(x_m)-x_m|\longrightarrow 0$ and $|f(y_m)-y_m|\longrightarrow 0$, then $|x_m-y_m|\longrightarrow 0$. $S'$ is open in a space with dimension lower than $n$.
\item For $\bar{\varepsilon}$ defined above there exists a finite $\varepsilon$-approximation $L=\{x^1, x^2, \dots, x^l\}$ to $\Delta$ for each $\varepsilon$ with $0<\varepsilon<\bar{\varepsilon}$ such that if for all sequences $(x_m)_{m\geq 1}$, $(y_m)_{m\geq 1}$ in each open $\varepsilon$-ball $S'$ around each $x^i\in L$ $|f(x_m)-x_m|\longrightarrow 0$ and $|f(y_m)-y_m|\longrightarrow 0$, then $|x_m-y_m|\longrightarrow 0$.
\end{enumerate}\label{sln}
\end{defn}
(1) of this definition is the same as (1) of the definition of modified local non-constancy. 

Now we show the following two lemmas.
\begin{lem}
Sequential local non-constancy means modified local non-constancy.\label{seq}
\end{lem}
The essence of this proof is according to the proof of Proposition 1 of \cite{berg}.
\begin{proof}
Let $S'$ be a set as defined in (2) or (3) of Definition \ref{sln}. Construct a sequence $(z(m))_{m\geq 1}$ in $S'$ such that $|f(z(m))-z(m)|\longrightarrow 0$. Consider $x$, $y$ in $S'$ with $x\neq y$. Construct an increasing binary sequence $(\lambda_m)_{m\geq 1}$ such that
\[\lambda_m=0\Rightarrow \max(|f(x)-x|, |f(y)-y|)<2^{-m},\]
\[\lambda_m=1\Rightarrow \max(|f(x)-x|, |f(y)-y|)>2^{-m-1}.\]
We may assume that $\lambda_1=0$. If $\lambda_m=0$, set $x(m)=x$ and $y(m)=y$. If $\lambda_m=1$, set $x(m)=y(m)=z(m)$. Now the sequences $(|f(x(m))-x(m)|)_{m\geq 1}$, $(|f(y(m))-y(m)|)_{m\geq 1}$ converge to 0, and so by sequential local non-constancy $|x(m)-y(m)|\longrightarrow 0$. Computing $M$ such that $|x(M)-y(M)|<|x-y|$, we see that $\lambda_M=1$. Therefore, $f(x)\neq x$ or $f(y)\neq y$.

Let $S$ be an open set in $\Delta$ or in a face of $\Delta$. Then there exists an $\varepsilon$-approximation to $\Delta$ or the face of $\Delta$ with sufficiently small $\varepsilon$ such that an $\varepsilon$-ball around some point in that $\varepsilon$-approximation is included in $S$. Therefore, there exists a point $x$ in $S$ such that $f(x)\neq x$.
\end{proof}

\begin{lem}
Let $f$ be a uniformly sequentially continuous function from $\Delta$ to itself, and assume that $\inf_{x\in S}|f(x)-x|=0$ where $S$ is nonempty and $S\subset \Delta$. If the following property holds:
\begin{quote}
For each $\varepsilon>0$ there exists $\eta>0$ such that if $x, y\in S$, $|f(x)-x|<\eta$ and $|f(y)-y|<\eta$, then $|x-y|\leq \varepsilon$.
\end{quote}
Then, there exists a point $\xi\in \Delta$ such that $f(\xi)=\xi$, that is, a fixed point of $f$. \label{fix0}
\end{lem}
\begin{proof}
Choose a sequence $(x(m))_{m\geq 1}$ in $S$ such that $|f(x(m))-x(m)|\longrightarrow 0$. Compute $M$ such that $|f(x(m))-x(m)|<\eta$ for all $m\geq M$. Then, for $l, m\geq M$ we have $|x(l)-x(m)|\leq \varepsilon$. Since $\varepsilon>0$ is arbitrary, $(x(m))_{m\geq 1}$ is a Cauchy sequence in S, and converges to a limit $\xi\in S$. The continuity of $f$ yields $|f(\xi)-\xi|=0$, that is, $f(\xi)=\xi$.
\end{proof}
The converse of Lemma \ref{seq} does not hold because the sequential local non-constancy implies isolatedness of fixed points but the local non-constancy and the modified local non-constancy do not.

Using Sperner's lemma (for modified partition of a simplex) we show that there exists an exact fixed point for any sequentially locally non-constant and uniformly sequentially continuous function from an $n$-dimensional simplex to itself. 
\begin{thm}[Brouwer's fixed point theorem for sequentially  locally non-constant functions]
Any sequentially locally non-constant and uniformly sequentially continuous function from an $n$-dimensional simplex $\Delta$ to itself has a fixed point.
\end{thm}
\begin{proof}
Let us prove this theorem through some steps.
\begin{enumerate}
\item First we show that we can partition $\Delta$ so that the conditions for Sperner's lemma (for modified partition of a simplex) are satisfied. We partition $\Delta$ according to the method in the proof of Sperner's lemma, and label the vertices of simplices constructed by partition of $\Delta$. It is important how to label the vertices contained in the faces of $\Delta$. Let $K$ be the set of small simplices constructed by partition of $\Delta$, $x=(x_0, x_1, \dots, x_n)$ be a vertex of a simplex of $K$, and denote the $i$-th coordinate of $f(x)$ by $f_i$ or $f(x)_i$. We label a vertex $x$ according to the following rule,
\[\mathrm{If}\ x_k>f_k,\ \mathrm{we\ label}\ x\ \mathrm{with}\ k.\]
If there are multiple $k$'s which satisfy this condition, we label $x$ conveniently for the conditions for Sperner's lemma to be satisfied.

Let us check labeling for vertices in three cases.
\begin{enumerate}
	\item Vertices of $\Delta$:

One of the coordinates of a vertex $x$ of $\Delta$ is 1, and all other coordinates are zero. Consider a vertex $(1, 0, \dots, 0)$. By the modified local non-constancy $f(x)\neq x$ means $f_j>x_j$ or $f_j<x_j$ for at least one $j$. $f_i<x_i$ can not hold for $i\neq 0$. On the other hand, $f_0>x_0$ can not hold. When $f_0(x)<x_0$, we label $x$ with 0. Assume that $f_i(x)>x_i=0$ for some $i\neq 0$. Then, since $\sum_{j=0}^nf_j(x)=1=x_0$, we have $f_0(x)<x_0$. Therefore, $x$ is labeled with 0. Similarly a vertex $x$ whose $k$-th coordinate is 1 is labeled with $k$ for all $k\in \{0, 1, \dots, n\}$.

	\item Vertices in the faces of $\Delta$:

Let $x$ be a vertex of a simplex contained in an $n-1$-dimensional face of $\Delta$ such that $x_i=0$ for one $i$ among $0, 1, 2, \dots, n$ (its $i$-th coordinate is 0). $f(x)\neq x$ means that $f_j>x_j$ or $f_j<x_j$ for at least one $j$. $f_i<x_i=0$ can not hold. When $f_k<x_k$ for some $k\neq i$, we label $x$ with $k$. Assume $f_i>x_i=0$. Then, since $\sum_{j=0}^nx_j=\sum_{j=0}^nf_j=1$, we have $f_k<x_k$ for some $k\neq i$, and we label $x$ with $k$. Assume that $f_j>x_j$ for some $j\neq i$. Then, since $x_i=0$ and
\[\sum_{l=0, l\neq j}^n f_l<\sum_{l=0, l\neq j}^nx_l,\]
we have $f_k<x_k$ for some $k\neq i, j$, and we label $x$ with $k$.

We have proved that we can label each vertex of a simplex contained in an $n-1$-dimensional face of $\Delta$ such that $x_i=0$ for one $i$ among $0, 1, 2, \dots, n$ with a number other than $i$. By similar procedures we can show that we can label the vertices of a simplex contained in an $n-2$-dimensional face of $\Delta$ such that $x_i=0$ for two $i$'s among $0, 1, 2, \dots, n$ with a number other than those $i$'s, and so on.

\begin{quote}
Consider a case where, for example, $x_{i+1}=x_{i+1}=0$. Neither $f_i<x_i=0$ nor $f_i<x_{i+1}=0$ can hold. When $f_k<x_k$ for some $j\neq i, i+1$, we label $x$ with $k$. Assume $f_i>x_i=0$ or $f_{i+1}>x_{i+1}=0$. Then, since $\sum_{j=0}^nx_j=\sum_{j=0}^nf_j=1$, we have $f_k<x_k$ for some $k\neq i,\ i+1$, and we label $x$ with $k$. Assume that $f_j>x_j$ for some $j\neq i, i+1$. Then, since $x_i=x_{i+1}=0$ and
\[\sum_{l=0, l\neq j}^n f_l<\sum_{l=0, l\neq j}^nx_l,\]
we have $f_k<x_k$ for some $k\neq i, i+1, j$, and we label $x$ with $k$.
\end{quote}

	\item Vertices of small simplices inside $\Delta$:

By the modified local non-constancy of $f$ every vertex $x$ in a modified partition of a simplex can be selected to satisfy $f(x)\neq x$. Assume that $f_i>x_i$ for some $i$. Then, since $\sum_{j=0}^n x_j=\sum_{j=0}^n f_j=1$, we have
\[f_k<x_k\]
for some $k\neq i$, and we label $x$ with $k$.
\end{enumerate}

Therefore, the conditions for Sperner's lemma (for modified partition of a simplex) are satisfied, and there exists an odd number of fully labeled simplices in $K$.

\item Consider a sequence $(\Delta_m)_{m\geq 1}$ of partitions of $\Delta$, and a sequence of fully labeled simplices $(\delta_m)_{m\geq 1}$. The larger $m$, the finer the partition. The larger $m$, the smaller the diameter of a fully labeled simplex. Let $x_m^0, x_m^1, \dots$ and $x_m^n$ be the vertices of a fully labeled simplex $\delta_m$. We name these vertices so that $x_m^0, x_m^1, \dots, x_m^n$ are labeled, respectively, with 0, 1, $\dots$, $n$. The values of $f$ at theses vertices are $f(x_m^0), f(x_m^1), \dots$ and $f(x_m^n)$. We can consider sequences of vertices of fully labeled simplices. Denote them by $(x_m^0)_{m\geq 1}$, $(x_m^1)_{m\geq 1}$, $\dots$, and $(x_m^n)_{m\geq 1}$. And consider sequences of the values of $f$ at vertices of fully labeled simplices. Denote them by $(f(x_m^0))_{m\geq 1}$, $(f(x_m^1))_{m\geq 1}$, $\dots$, and $(f(x_m^n))_{m\geq 1}$. By the uniform sequential continuity of $f$
\[|(f(x_m^i))_{m\geq 1}-(f(x_m^j))_{m\geq 1}|\longrightarrow 0\ \mathrm{whenever}\ |(x_m^i)_{m\geq 1}-(x_m^j)_{m\geq 1}|\longrightarrow 0,\]
for $i\neq j$. $|(x_m^i)_{m\geq 1}-(x_m^j)_{m\geq 1}|\longrightarrow 0$ means
\[\forall \varepsilon>0\ \exists M\ \forall m\geq M\ (|x_m^i-x^j_m|<\varepsilon)\ i\neq j,\]
and $|(f(x_m^i))_{m\geq 1}-(f(x_m^j))_{m\geq 1}|\longrightarrow 0$ means
\[\forall \varepsilon>0\ \exists M'\ \forall m\geq M'\ (|f(x^i_m)-f(x^j_m)|<\varepsilon)\ i\neq j.\]
Consider a fully labeled simplex $\delta_l$ in partition of $\Delta$ such that $l\geq \max(M, M')$. Denote vertices of $\delta_l$ by $x^0$, $x^1$, $\dots$, $x^n$. We name these vertices so that $x^0, x^1, \dots, x^n$ are labeled, respectively, with 0, 1, $\dots$, $n$. Then, $|x^i-x^j|<\varepsilon$ and $|f(x^i)-f(x^j)|<\varepsilon$.

About $x^0$, from the labeling rules we have $x^0_0>f(x^0)_0$. About $x^1$, also from the labeling rules we have $x^1_1>f(x^1)_1$ which implies $x^1_1>f(x^1)_1$. $|f(x^0)-f(x^1)|<\varepsilon$ means $f(x^1)_1>f(x^0)_1-\varepsilon$. On the other hand, $|x^0-x^1|<\varepsilon$ means $x^0_1>x^1_1-\varepsilon$. Thus, from
\[x^0_1>x^1_1-\varepsilon,\ x^1_1>f(x^1)_1,\ f(x^1)_1>f(x^0)_1-\varepsilon\]
we obtain
\[x^0_1>f(x^0)_1-2\varepsilon\]
By similar arguments, for each $i$ other than 0,
\begin{equation}
x^0_i>f(x^0)_i-2\varepsilon. \label{e1}
\end{equation}
For $i=0$ we have
\begin{equation}
x^0_0>f(x^0)_0 \label{e2}
\end{equation}
Adding (\ref{e1}) and (\ref{e2}) side by side except for some $i$ (denote it by $k$) other than 0,
\[\sum_{j=0, j\neq k}^{n} x^0_j>\sum_{j=0, j\neq k}^{n} f(x^0)_j-2(n-1)\varepsilon.\]
From $\sum_{j=0}^{n} x^0_j=1$, $\sum_{j=0}^{n} f(x^0)_j=1$ we have $1-x^0_k>1-f(x^0)_k-2(n-1)\varepsilon$, which is rewritten as
\[x^0_k<f(x^0)_k+2(n-1)\varepsilon.\]
Since (\ref{e1}) implies $x^0_k>f(x^0)_k-2\varepsilon$, we have
\[f(x^0)_k-2\varepsilon<x^0_k<f(x^0)_k+2(n-1)\varepsilon.\]
Thus,
\begin{equation}
|x^0_k-f(x^0)_k|<2(n-1)\varepsilon \label{e3}
\end{equation}
is derived. On the other hand, adding (\ref{e1}) from 1 to $n$ yields
\begin{equation*}
\sum_{j=1}^{n} x^0_j>\sum_{j=1}^{n} f(x^0)_j-2n\varepsilon.
\end{equation*}
From $\sum_{j=0}^{n} x^0_j=1$, $\sum_{j=0}^{n} f(x^0)_j=1$ we have
\begin{equation}
1-x^0_0>1-f(x^0)_0-2n\varepsilon.\label{nash1}
\end{equation}
Then, from (\ref{e2}) and (\ref{nash1}) we get
\begin{equation}
|f(x^0)_0-x^0_0|<2n\varepsilon. \label{e24}
\end{equation}
From (\ref{e3}) and (\ref{e24}) we obtain the following result,
\begin{equation*}
|f(x^0)_i-x^0_i|<2n\varepsilon\ \mathrm{for\ all}\ i. 
\end{equation*}
Thus,
\begin{equation}
|f(x^0)-x^0|<n(n+1)(2\varepsilon).\label{fp}
\end{equation}
Since $\varepsilon$ is arbitrary and $n$ is finite, $\inf_{x\in \delta^n}|f(x)-x|=0$.

\item Choose a sequence $(\xi(m))_{m\geq 1}$ in $\delta^n$ such that $|f(\xi(m))-\xi(m)|\longrightarrow 0$. In view of Lemma \ref{fix0} it is enough to prove that the following condition holds.
\begin{quote}
For each $\varepsilon>0$ there exists $\eta>0$ such that if $x, y\in \delta^n$, $|f(x)-x|<\eta$ and $|f(y)-y|<\eta$, then $|x-y|<\varepsilon$.
\end{quote}
Assume that the set
\[K=\{(x,y)\in \delta^n\times \delta^n:\ |x-y|\geq \varepsilon\}\]
is nonempty and compact. Since the mapping $(x,y)\longrightarrow \max(|f(x)-x|,|f(y)-y|)$ is uniformly continuous, we can construct an increasing binary sequence $(\lambda(m))_{m\geq 1}$ such that
\[\lambda_m=0\Rightarrow \inf_{(x,y)\in K}\max(|f(x)-x|,|f(y)-y|)<2^{-m},\]
\[\lambda_m=1\Rightarrow \inf_{(x,y)\in K}\max(|f(x)-x|,|f(y)-y|)>2^{-m-1}.\]
It suffices to find $m$ such that $\lambda_m=1$. In that case, if $|f(x)-x|<2^{-m-1}$ and $|f(y)-y|<2^{-m-1}$, we have $(x,y)\notin K$ and $|x-y|\leq \varepsilon$. Assume $\lambda_1=0$. If $\lambda_m=0$, choose $(x(m), y(m))\in K$ such that $\max(|f(x(m))-x(m)|, |f(y(m))-y(m)|)<2^{-m}$, and if $\lambda_m=1$, set $x(m)=y(m)=\xi(m)$. Then, $|f(x(m))-x(m)|\longrightarrow 0$ and $|f(y(m))-y(m)|\longrightarrow 0$, so $|x(m)-y(m)|\longrightarrow 0$. Computing $M$ such that $|x(M)-y(M)|<\varepsilon$, we must have $\lambda_M=1$. Note that $f$ is a sequentially locally non-constant and uniformly sequentially continuous function from $\Delta$ to itself. Thus, we have completed the proof.
\end{enumerate}
\end{proof}

\section{From Brouwer's fixed point theorem for sequentially locally non-constant and uniformly sequentially continuous functions to Sperner's lemma}

In this section we will derive Sperner's lemma from Brouwer's fixed point theorem for sequentially locally non-constant and uniformly sequentially continuous functions. Let $\Delta$ be an $n$-dimensional simplex. Denote a point on $\Delta$ by $x$. Consider a function $f$ from $\Delta$ to itself. Partition $\Delta$ in the way depicted in Figure \ref{tria2} in the appendix. Let $K$ denote the set of small $n$-dimensional simplices of $\Delta$ constructed by partition. Vertices of these small simplices of $K$ are labeled with the numbers 0, 1, 2, $\dots$, $n$ subject to the same rules as those in Lemma \ref{l2}. Now we derive Sperner's lemma expressed in Lemma \ref{l2} from Brouwer's fixed point theorem for sequentially locally non-constant and uniformly sequentially continuous functions.

Denote the vertices of an $n$-dimensional simplex of $K$ by $x^0, x^1, \dots, x^n$, the $j$-th coordinate of $x^i$ by $x^i_j$, and denote the label of $x^i$ by $l(x^i)$. Let $\tau$ be a positive number which is smaller than $x^i_{l(x^i)}$ for all $x^i$, and define a function $f(x^i)$ as follows\footnote{We refer to \cite{yoseloff} about the definition of this function.}:
\[f(x^i)=(f_0(x^i), f_1(x^i), \dots, f_n(x^i)),\]
and
\begin{equation}
f_j(x^i)=\left\{
\begin{array}{ll}
x^i_j-\tau&\mathrm{for}\ j=l(x^i),\\
x^i_j+\frac{\tau}{n}&\mathrm{for}\ j\neq l(x^i).\label{e0}
\end{array}
\right.
\end{equation}
$f_j$ denotes the $j$-th component of $f$. From the labeling rules we have $x^i_{l(x^i)}>0$ for all $x^i$, and so $\tau>0$ is well defined. Since $\sum_{j=0}^nf_j(x^i)=\sum_{j=0}^nx^i_j=1$, we have
\[f(x^i)\in \Delta.\]
We extend $f$ to all points in the simplex by convex combinations on the vertices of the simplex. Let $z$ be a point in the $n$-dimensional simplex of $K$ whose vertices are $x^0, x^1, \dots, x^n$. Then, $z$ and $f(z)$ are expressed as follows:
\[z=\sum_{i=0}^n\lambda_ix^i,\ \mathrm{and}\ f(z)=\sum_{i=0}^n\lambda_if(x^i),\ \lambda_i\geq 0,\ \sum_{i=0}^n\lambda_i=1.\]

We very that $f$ is uniformly sequentially continuous. Consider sequences $(x(n))_{n\geq 1}$, $(x'(n))_{n\geq 1}$, $(f(x(n)))_{n\geq 1}$ and $(f(x'(n)))_{n\geq 1}$ such that $|x(n)-x'(n)|\longrightarrow 0$. Denote each component of $x(n)$ by $x(n)_j$ and so on. When $|x(n)-x'(n)|\longrightarrow 0$, $|x(n)_j-x'(n)_j|\longrightarrow 0$ for each $j$. Then, since $\tau>0$, we have $|f(x(n))-f(x'(n))|\longrightarrow 0$, and so $f$ is uniformly sequentially continuous.

Next we verify that $f$ is sequentially locally non-constant.
\begin{enumerate}
	\item Assume that a point $z$ is contained in an $n-1$-dimensional small simplex $\delta^{n-1}$ constructed by partition of an $n-1$-dimensional face of $\Delta$ such that its $i$-th coordinate is $z_i=0$. Denote the vertices of $\delta^{n-1}$ by $z^j,\ j=0, 1, \dots, n-1$ and their $i$-th coordinate by $z_i^j$. Then, we have
\[f_i(z)=\sum_{j=0}^{n-1}\lambda_jf_i(z^j),\ \lambda_j\geq 0,\ \sum_{j=0}^n\lambda_j=1.\]
Since all vertices of $\delta^{n-1}$ are not labeled with $i$, (\ref{e0}) means $f_i(z^j)>z^j_i$ for all $j=\{0, 1, \dots, n-1\}$. Then, there exists no sequence $(z(m))_{m\geq 1}$ such that $|f(z(m))-z(m)|\longrightarrow 0$ in an $n-1$-dimensional face of $\Delta$.
\item Let $z$ be a point in an $n$-dimensional simplex $\delta^n$. Assume that no vertex of $\delta^n$ is labeled with $i$. Then
\begin{equation}
f_i(z)=\sum_{j=0}^n\lambda_jf_i(x^j)=z_i+\left(1+\frac{1}{n}\right)\tau.\label{e8}
\end{equation}
Then, there exists no sequence $(z(m))_{m\geq 1}$ such that $|f(z(m))-z(m)|\longrightarrow 0$ in $\delta^n$.
\item Assume that $z$ is contained in a fully labeled $n$-dimensional simplex $\delta^n$, and rename vertices of $\delta^n$ so that a vertex $x^i$ is labeled with $i$ for each $i$. Then,
\begin{align*}
f_i(z)=\sum_{j=0}^n\lambda_jf_i(x^j)=\sum_{j=0}^n\lambda_jx_i^j+\sum_{j\neq i}\lambda_j\frac{\tau}{n}-\lambda_i\tau=z_i+\left(\frac{1}{n}\sum_{j\neq i}\lambda_j-\lambda_i\right)\tau\ \mathrm{for\ each}\ i.
\end{align*}
Consider sequences $(z(m))_{m\geq 1}=(z(1), z(2), \dots)$, $(z'(m))_{m\geq 1}=(z'(1), z'(2), \dots)$ such that $|f(z(m))-z(m)|\longrightarrow 0$ and $|f(z'(m))-z'(m)|\longrightarrow 0$.

Let $z(m)=\sum_{i=0}^n\lambda(m)_ix^i$ and $z'(m)=\sum_{i=0}^n\lambda'(m)_ix^i$. Then, we have
\[\frac{1}{n}\sum_{j\neq i}\lambda(m)_j-\lambda(m)_i\longrightarrow 0,\ \mathrm{and}\ \frac{1}{n}\sum_{j\neq i}\lambda'(m)_j-\lambda'(m)_i\longrightarrow 0\ \mathrm{for\ all}\ i.\]
Therefore, we obtain
\[\lambda(m)_i\longrightarrow \frac{1}{n+1}, \mathrm{and}\ \lambda'(m)_i\longrightarrow \frac{1}{n+1}.\]
These mean
\[|z(m)-z'(m)|\longrightarrow 0.\]
\end{enumerate}
Thus, $f$ is sequentially locally non-constant, and it has a fixed point. Let $z^*$ be a fixed point of $f$. We have
\begin{equation}
z^*_i=f_i(z^*)\ \mathrm{for\ all}\ i.\label{e15}
\end{equation}
Suppose that $z^*$ is contained in a small $n$-dimensional simplex $\delta^*$. Let $z^0, z^1, \dots, z^n$ be the vertices of $\delta^*$. Then, $z^*$ and $f(z^*)$ are expressed as
\[z^*=\sum_{i=0}^n\lambda_iz^i\ \mathrm{and}\ f(z^*)=\sum_{i=0}^n\lambda_if(z^i),\ \lambda_i\geq 0,\ \sum_{i=0}^n\lambda_i=1.\]
(\ref{e0}) implies that if only one $z^k$ among $z^0, z^1, \dots, z^n$ is labeled with $i$, we have
\[f_i(z^*)=\sum_{j=0}^n\lambda_jf_i(z^j)=\sum_{j=0}^n\lambda_jz_i^j+\sum_{j\neq k}^n\lambda_j\frac{\tau}{n}-\lambda_k\tau=z_i^*\ \mathrm{(}z_i^*\mathrm{\ is\ the}\ i\mathrm{-th\ coordinate\ of}\ z^*\mathrm{)}.\]
This means
\[\frac{1}{n}\sum_{j\neq k}^n\lambda_j-\lambda_k=0.\]
Then, (\ref{e15}) is satisfied with $\lambda_k=\frac{1}{n+1}$ for all $k$. If no $z^j$ is labeled with $i$, we have (\ref{e8}) with $z=z^*$ and then (\ref{e15}) can not be satisfied. Thus, one and only one $z^j$ must be labeled with $i$ for each $i$. Therefore, $\delta^*$ must be a fully labeled simplex, and so the existence of a fixed point of $f$ implies the existence of a fully labeled simplex.

We have completely proved Sperner's lemma.

\section{Concluding Remarks}

As a future research program we are studying the following themes.
\begin{enumerate}
	\item An application of Brouwer's fixed point theorem for sequentially locally non-constant functions to economic theory and game theory, in particular, the problem of the existence of an equilibrium in a competitive economy with excess demand function with property that is similar to sequential local non-constancy, and the existence of a Nash equilibrium in a strategic game with payoff functions which satisfy the property of sequential local non-constancy.
	\item A generalization of the result of this paper to Kakutani's fixed point theorem for multi-valued functions with property of sequential local non-constancy and its application to economic theory.
\end{enumerate}

\appendix
\section{Proof of Sperner's lemma}\label{app1}
We prove Sperner's lemma by induction about the dimension of $\Delta$. When $n=0$, we have only one point with the number 0. It is the unique 0-dimensional simplex. Therefore the lemma is trivial. When $n=1$, a partitioned 1-dimensional simplex is a segmented line. The endpoints of the line are labeled distinctly, by 0 and 1. Hence in moving from endpoint 0 to endpoint 1 the labeling must switch an odd number of times, that is, an odd number of edges labeled with 0 and 1 may be located in this way.

\begin{figure}[ht!]
\begin{center}
\includegraphics{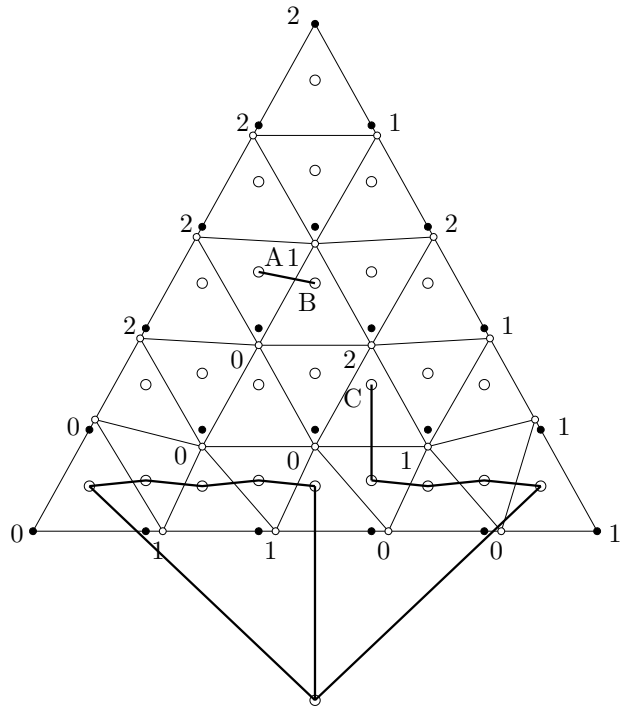}
\end{center}
	\vspace*{-.3cm}
	\caption{Sperner's lemma}
	\label{tria2}
\end{figure}

Next consider the case of 2 dimension. Assume that we have partitioned a 2-dimensional simplex (triangle) $\Delta$ as explained above. Consider the face of $\Delta$ labeled with 0 and 1\footnote{We call edges of triangle $\Delta$ \emph{faces} to distinguish between them and edges of a dual graph which we will consider later.}. It is the base of the triangle in Figure \ref{tria2}. Now we introduce a dual graph that has its nodes in each small triangle of $K$ plus one extra node outside the face of $\Delta$ labeled with 0 and 1 (putting a dot in each small triangle, and one dot outside $\Delta$). We define edges of the graph that connect two nodes if they share a side labeled with 0 and 1. See Figure \ref{tria2}. White circles are nodes of the graph, and thick lines are its edges. Since from the result of 1-dimensional case there are an odd number of faces of $K$ labeled with 0 and 1 contained in the face of $\Delta$ labeled with 0 and 1, there are an odd number of edges which connect the outside node and inside nodes. Thus, the outside node has odd degree. Since by the Handshaking lemma there are an even number of nodes which have odd degree, we have at least one node inside the triangle which has odd degree. Each node of our graph except for the outside node is contained in one of small triangles of $K$. Therefore, if a small triangle of $K$ has one face labeled with 0 and 1, the degree of the node in that triangle is 1: if a small triangle of $K$ has two such faces, the degree of the node in that triangle is 2, and if a small triangle of $K$ has no such face, the degree of the node in that triangle is 0. Thus, if the degree of a node is odd, it must be 1, and then the small triangle which contains this node is labeled with 0, 1 and 2 (fully labeled). In Figure \ref{tria2} triangles which contain one of the nodes $A$, $B$, $C$ are fully labeled triangles.

Now assume that the theorem holds for dimensions up to $n-1$. Assume that we have partitioned an $n$-dimensional simplex $\Delta$. Consider the fully labeled face of $\Delta$ which is a fully labeled $n-1$-dimensional simplex. Again we introduce a dual graph that has its nodes in small $n$-dimensional simplices of $K$ plus one extra node outside the fully labeled face of $\Delta$ (putting a dot in each small $n$-dimensional simplex, and one dot outside $\Delta$). We define the edges of the graph that connect two nodes if they share a face labeled with 0, 1, $\dots$, $n-1$. Since from the result of $n-1$-dimensional case there are an odd number of fully labeled faces of small simplices of $K$ contained in the $n-1$-dimensional fully labeled face of $\Delta$, there are an odd number of edges which connect the outside node and inside nodes. Thus, the outside node has odd degree. Since, by the Handshaking lemma there are an even number of nodes which have odd degree, we have at least one node inside the simplex which has odd degree. Each node of our graph except for the outside node are contained in one of small $n$-dimensional simplices of $K$. Therefore, if a small simplex of $K$ has one fully labeled face, the degree of the node in that simplex is 1: if a small simplex of $K$ has two such faces, the degree of the node in that simplex is 2, and if a small simplex of $K$ has no such face, the degree of the node in that simplex is 0. Thus, if the degree of a node is odd, it must be 1, and then the small simplex which contains this node is fully labeled.

\begin{quote}
If the number (label) of a vertex other than vertices labeled with 0, 1, $\dots$, $n-1$ of an $n$-dimensional simplex which contains a fully labeled $n-1$-dimensional face is $n$, then this $n$-dimensional simplex has one such face, and this simplex is a fully labeled $n$-dimensional simplex. On the other hand, if the number of that vertex is other than $n$, then the $n$-dimensional simplex has two such faces.
\end{quote}

We have completed the proof of Sperner's lemma.

Since $n$ and partition of $\Delta$ are finite, the number of small simplices constructed by partition is also finite. Thus, we can constructively find a fully labeled $n$-dimensional simplex of $K$ through finite steps.


\begin{thebibliography}{9999}

\bibitem{berg} J. Berger, D. Bridges and P. Schuster, ``The Fan Theorem and unique existence of maxima'', \textit{Hournal of Symbolic Logic}, vol. 71, pp. 713-720, 2006.
\bibitem{bb} E. Bishop and D. Bridges, \textit{Constructive Analysis}, Springer, 1985.
\bibitem{br2} D. Bridges, Ishihara H., Schuster, P. and V\^{i}\c{t}\u{a}, L., ``Strong continuity implies uniform sequential continuity'', \textit{Archive for Mathematical Logic}, vol. 44, pp. 887-895, 2005.
\bibitem{bm} D. Bridges and R. Mines, ``Sequentially Continuous Linear Mappings in Constructive Analysis'', \textit{Journal of Symbolic Logic}, vol. 63, pp. 579-583.
\bibitem{br} D. Bridges and F. Richman, \textit{Varieties of Constructive Mathematics}, Cambridge University Press, 1987.
\bibitem{bv} D. Bridges and L. V\^{i}\c{t}\u{a}, \textit{Techniques of Constructive Mathematics}, Springer, 2006.
\bibitem{da} D. van Dalen, ``Brouwer's $\varepsilon$-fixed point from Sperner's lemma'', \textit{Logic Group Preprint Series}, No. 275, 2009.
\bibitem{ishi} H. Ishihara, ``Continuity properties in constructive mathematics'', \textit{Journal of Symbolic Logic}, vol. 57, pp. 557-565, 1992.
\bibitem{su} F. E. Su, ``Rental harmony: Sperner's lemma for fair devision'', \textit{American Mathematical Monthly}, vol. 106, pp. 930-942, 1999.

\bibitem{ta1} Y. Tanaka, ``Constructive proof of Brouwer's fixed point theorem for sequentially locally non-constant functions by Sperner's lemma
'',  arXiv:1103.1776v1.
\bibitem{veld} W. Veldman, ``Brouwer's approximate fixed point theorem is equivalent to Brouwer's fan theorem'', in \textit{Logicism, Intuitionism and Formalism}, edited by Lindstr\"{o}m, S., Palmgren, E., Segerberg, K. and Stoltenberg-Hansen, Springer, 2009.
\bibitem{yoseloff} M. Yoseloff, ``Topological proofs of some combinatorial theorems'', \textit{Journal of Combinatorial Theory (A)}, vol. 17, pp. 95-111, 1974.

\end{thebibliography}
\end{document}